\documentclass[a4paper,11pt]{amsart} 
\usepackage{amscd,amssymb,amsmath,amsthm}
\usepackage{amsfonts}
\usepackage{graphicx,epsf}
\usepackage{enumerate,color}
\usepackage{a4wide}
\usepackage[mathcal]{euscript}

\usepackage{epstopdf}
\usepackage{hyperref}
\usepackage{cleveref}
\usepackage{epsfig}

  \usepackage{algorithm}
  \usepackage{algpseudocode}
  \newcommand*\Let[2]{\State #1 $\gets$ #2}

\newtheorem{theorem}{Theorem}[section]
\newtheorem{lemma}[theorem]{Lemma}
\newtheorem{proposition}[theorem]{Proposition}

\newcommand{\R}{\mathbb{R}}

\newcommand{\Pro}{\mathbb{P}}

\newcommand{\vecu}{\text{u}}
\newcommand{\vecm}{\text{m}}

\newcommand{\conu}{U}
\newcommand{\conv}{V}
\newcommand{\conf}{M}

\DeclareMathOperator*{\argmin}{arg\,min}

\usepackage{todonotes}

\begin{document}

	\title[Variational Reconstruction for Dynamic X-ray Tomography]{A Variational Reconstruction Method for Undersampled Dynamic X-ray Tomography based on Physical Motion Models}
	\author[M. Burger, H. Dirks, L. Frerking, A. Hauptmann, T. Helin, and S. Siltanen]{}
	\date{\today}
\maketitle

\centerline{\scshape Martin Burger$^{\rm a}$, Hendrik Dirks$^{\rm a}$, Lena Frerking$^{\rm a}$, Andreas Hauptmann$^{\rm b}$}
\centerline{\scshape Tapio Helin$^{\rm c}$, and Samuli Siltanen$^{\rm c}$}
\medskip
\centerline{\footnotesize \noindent$^{\rm a}$Institut f\"ur Numerische und Angewandte Mathematik, Westf\"alische Wilhelms-Universit\"at (WWU)
M\"unster, M\"unster, Germany.} 
\centerline{\footnotesize \noindent$^{\rm b}$Department of Computer Science, University College London, London, United Kingdom }
\centerline{\footnotesize \noindent$^{\rm c}$Department of Mathematics and Statistics, University of Helsinki, Helsinki, Finland }

\begin{abstract}
In this paper we study the reconstruction of moving object densities from undersampled dynamic X-ray tomography in two dimensions. A particular motivation of this study is to use realistic measurement protocols for practical applications, i.e. we do not assume to have a full Radon transform in each time step, but only projections in few angular directions. This restriction enforces a space-time reconstruction, which we perform by  incorporating physical motion models and regularization of motion vectors in a variational framework. The methodology of optical flow, which is one of the most common methods to estimate motion between two images, is utilized to formulate a joint variational model for reconstruction and motion estimation. 
 
We provide a basic mathematical analysis of the forward model and the variational model for the image reconstruction. Moreover, we discuss the efficient numerical minimization based on alternating minimizations between images and motion vectors. A variety of results are presented for simulated and real measurement data with different sampling strategy. A key observation is that random sampling combined with our model allows reconstructions of similar amount of measurements and quality as a single static reconstruction.
\end{abstract}


\section{Introduction}\label{sec:intro}

Tomographic image reconstruction is probably the most popular inverse problem, with a mathematical history of now 100 years after the seminal work of Radon \cite{Radon1917} and enormous impact on applications for more than 50 years \cite{Cormack1963,Shepp1974,Smith1977}. Nonetheless, it still offers a wealth of mathematical problems, most of them being driven by concrete practical issues, e.g. limited angles or limited field-of-view. A question that receives particular attention in the last years is dynamic tomography, i.e. time-dependent projection measurements of a non-stationary object. Of particular interest are moving objects, e.g. organs in medical computerized tomography (CT), which can lead to significant artifacts in a stationary reconstruction even in the case of slow motion. There have been several studies to tackle this problem, for instance by gating, where projections are either only taken at specific time instances of a periodic movement or by choosing only those projections after the measurement, typically used for imaging of the beating heart \cite{Achenbach2001,Trabold2003}. More recent studies estimate the motion to perform a motion compensated reconstruction \cite{Blondel2004,Li2005}. The problem of motion compensation is analyzed in a rigorous mathematical framework in \cite{Hahn2014,Hahn2016a,Hahn2016}. However, all of the techniques above aim to reconstruct a static CT image from dynamic data. In this work we are rather interested in the reconstruction of the object and the dynamics in space and time, and hence are not limited to specific movements. For this purpose we discuss a joint variational approach that incorporates physical motion models and regularization of motion vectors to achieve improved reconstruction results. We use the well-known optical flow constraint in spatial dimension two, i.e. we assume that the intensity of the images we aim to reconstruct is constant over time. The approach is however extendable to three-dimensional density reconstruction, where an even more physical modelling with a continuity equation constraining the dynamics is possible, reminiscent of optimal transport type approaches (cf. \cite{Bru10}).

A particular goal of our study is to use realistic measurement protocols for practical applications, i.e. we do not assume to have a full Radon transform in each time step, but only projections in few angular directions. Obviously in real life tomographs usually acquire one angular direction after the other, so one would need to work with single angles in an ideal modeling. However, for a suitable mathematical model we can compare the multiple time scales appearing during the process: the scale $\tau_P$ needed to take a projection at fixed angle, the typical time scale $\tau_R$ to perform a rotation, and finally the time scale $\tau_M$ of the object motion. The latter can be determined as a ratio of the spatial size of the objects one is interested in and their speed. If one finds that one or two of those time scales are smaller than others by magnitude, they can usually be ignored. For slowly moving objects with $\tau_M \gg \tau_P$ and $\tau_M \gg \tau_R$, it is indeed realistic to assume that a full (or limited but not small) set of angles can be acquired in each time step. Even in this case one might be limited by other factors however, such as dose considerations in medical X-ray tomography,prohibiting to acquire multiple full X-rays consecutively. Thus, we will focus on the important case of few angles per time step, which does not allow to perform separate static reconstructions at single steps, but indeed enforces to perform space-time reconstruction. Without additional prior information on the dynamics, the latter is highly underdetermined and hence we shall incorporate physical motion models into variational regularization methods, which will be based on recently proposed motion corrected image reconstruction techniques, cf. \cite{Dir15, dirks2016joint, dirks2016multiframe, Frerking2016,suhr2015variational}. A preliminary application similar to ours is discussed in \cite{Huynh2017}. In order to deal with the different possibilities to measure few angles, we use a time-dependent forward operator, which is also the main change in the method compared to \cite{Dir15}. Hence, we will keep the analysis in this paper rather short and mainly highlight the needed modifications in a time-continuous setting. Moreover, we discuss different time discretizations induced by measurement times and the corresponding time-discrete motions. 

The main focus of the paper is the computational side and a detailed comparison of possible results in different measurement (sampling) setups, restricting ourselves to a two-dimensional setup (one projection being a single line integral), which allows to gain good insight into the problem. We will consider the following realistic cases:
\begin{itemize}
\item {\bf Small angular increments: } In some cases, the object of interest can be very dynamic in relation to the rotation time scale, i.e. $\tau_R \gg \tau_M$. Then, the only way to obtain information regarding the dynamics is to apply small rotations, e.g. small increments between the angles in consecutive time steps. If $\tau_P < \tau_M$, a sufficiently small angular increment usually achieves $\tau_R \sim \tau_M$. Since the object of interest can evolve quite a lot before $180$ degrees of rotation are reached, it is expected that even with good motion models the reconstructions suffer from similar artefacts as in static limited angle tomography. Cardiac imaging with a single-source CT scanner provides an example of this case \cite{Ritman2003}. Problems with motion artifacts are typically overcome using {\it gating}, or imaging over several heartbeats and synchronizing the data with the help of a electrocardiogram. However, the proposed motion model approach paves the way towards dynamic tomography of one-shot events, such as the entry of contrast agent into bloodstream in angiography.

\item {\bf Small angular increments with multiple angles: } A modification of the previous setup is obtained if we consider small angular increments but with $k$ different projection angles at each time step with relative angles of $180/k$ degrees. This is a model for multi-source imaging systems, such as the classical ``dynamic spatial reconstructor'' with $k=23$ \cite{Robb1983}, or more recent dual-source or triple-source CT scanners \cite{Flohr2006,Zhao2010}. Another possibility is to consider the case of $\tau_R + \tau_P \ll \tau_M$ even for larger angles, but a dose constraint allowing only $k$ angles per time step.

\item {\bf Tracking: } Additionally, we consider a case with a possibly different number of angles measured per time step, taking the extreme case of tracking by starting with a full data set and then acquiring a single angle over several time steps until the next full data set is obtained. One motivation for this approach is again dose limitation. Another one may be processes with inherently different time scales in the dynamics, a fast part that only allows to take single angles with small increments interchanging with a slow part such that $\tau_M \gg \tau_R$. For example, consider studying fluid flow in porous media using synchrotron radiation \cite{Buchi2008}. Before introducing fluid, the sample is static ($\tau_M =\infty$) and can be accurately imaged. During fluid flow we have a fast period with $\tau_R \gg \tau_M$. When the voids in the sample are fully occupied by fluid, we again have $\tau_M =\infty$ and can take a final accurate scan.

\item {\bf Randomized angles: } If $\tau_R \ll \tau_M$ and one has dose limitations (or if $\tau_P \gg \tau_R$) one can instead consider a setup with arbitrarily different angles in each time step, which is expected to improve the reconstruction quality if we can obtain a better sampling of the full $180$ degrees in smaller time. Taking into account recent results on randomized measurements in compressed sensing (cf. \cite{jorgensen2015empirical,koesters2017}) we consider a setup of randomized angles. For simplicity we restrict ourselves to choosing a single angle in each time step from a uniform distribution (independent from the other time steps), which already yields strongly improved results. This measurement setup could be implemented using arrays of individually flashable small X-ray emitters \cite{Zhang2006}.
\end{itemize}

In order to perform computational experiments, we use well-designed software phantoms as well as a hardware phantom measured with a custom-built $\mu$CT system at the University of Helsinki, see \cite{Bubba2016} for technical specifications. As physical target we choose three small ($\sim$25 mm$^2$) square ceramic stones that are centered close to the detector. The geometry is approximately parallel beam with a focus-to-detector distance of 630 mm.
The collected data consists of 30 stop-and-go measurements with $60$ different angles acquired in each times step, which can be considered as a full CT and hence provides some reference reconstructions. All measurement setups discussed above can be obtained by choosing a subset of the collected angles.

The remainder of the paper is organized as follows: in Section \ref{sec:modelling} we introduce a time-dependent Radon transform and formulate the reconstruction procedure in a time continuous setting. For estimating the motion we discuss the optical flow constraint and combine both models to a joint problem for image reconstruction and motion estimation. Subsequently, we present a possibility to analyze the numerical error by writing the dynamic system as a state-space model and applying Bayesian inference. In Section \ref{sec:numerics} we discuss the discretization of our model as well as practical issues to solve the optimization problem. Results of the proposed method are then presented in Section \ref{sec:experiments} for a simulated and a physical phantom. We conclude this study with an outlook to future work in Section \ref{sec:conclusions}.

\section{Variational models for dynamic tomography with motion}\label{sec:modelling}

In order to introduce variational models for the reconstruction of dynamic X-ray tomography data, we start by shortly recapping some basic properties that are needed. At first we define a time-dependent version of the Radon transform, which is applicable to dynamic data sets and show its well-definedness. Subsequently, we concentrate on one of the most popular approaches for estimating motion, which is the optical flow methodology. In principle it is also possible to use different motion models, e.g. nonlinear models to consider large scale movements, or continuity equations in a 3D space. However, in this paper we focus on the 2D optical flow model. We combine the optical flow approach with the time-dependent Radon transform and hence obtain a model for motion corrected variational reconstruction. In the end of the chapter we furthermore perform an uncertainty quantification for a state-space formulation of the introduced model.

\subsection{X-ray tomography and Radon transform}\label{sec:RadonTrafo}
In dynamic X-ray tomography one seeks to determine a time-dependent function $u(x,t)$, where $u$ models the non-negative absorption of photons on a bounded domain $\Omega\subset \R^2$ at a time instance $t\in\R_+$, hence $u:\Omega\times\R_+\to\R_+$. We consider the measurement model
\begin{equation}\label{eqn:measMod}
\mathcal{A}u=m,
\end{equation}
where the operator $\mathcal{A}$ denotes a time-dependent two-dimensional Radon transform 
\begin{equation}\label{eqn:RadonTransform}
(\mathcal{A}u)(\theta,s,t) := (\mathcal{R}_{I(t)}u(x,t))(\theta,s)= \int_{x\cdot \theta = s}u(x,t)\,dx.
\end{equation}
Here, $I(t)$ indicates the set of given measurements at time $t$. 
The full Radon transform maps to a parametrization of the infinite unit cylinder denoted by $Z^2:=\{(\theta,s): \theta\in S^1, s\in \R\}$, where ${S}^1$ is the unit circle in $\R^2$.
For a fixed time $t\geq0$ the obtained measurement $m(\theta,s;t)=(\mathcal{R}_{I(t)}u(x,t))(\theta,s)$ is called the sinogram, which consists of line integrals over $\Omega$ with respect to the set $I(t)$. The attenuation $u(x,t)$ at each time instance can be uniquely determined if one has knowledge of the full sinogram for all possible lines, as shown by Radon \cite{Radon1917}, see also \cite{Natterer1986}. However, we are interested in situations of undersampling, such that rather the full collection $m$ corresponds to the usual sinogram. This is apparent if there is a single angle, i.e. a unique $\theta=\hat \theta(t)$ for each $t$. Then the measurement is actually $\tilde m(s;t) = (\mathcal{R}_{I(t)}u(x,t))(\hat \theta(t),s)$

In order to give a sound definition of the time dependent Radon transform we introduce some additional notation and assumptions. In the following we assume that the time interval is fixed and bounded by the end point $T>0$. Furthermore, we assume that the measured angles might change between time steps and denote $I(t)\subset Z^2$ as the set of active measurement parameters in each time instance. We equip the set of measurement parameters $I(t)$ measured at time $t$ with a nonnegative Radon measure $\sigma_t$, noticing that in the undersampling situations we are interested in $\sigma_t$ will be a partially discrete measure with respect to $\theta$ for each $t$. We denote by $\Theta(t)$ the set of all $\theta$ appearing in $I(t)$.
 Then, with $\Omega$ denoting the bounded support of $u$ we consider for $I(t) \subset Z^2$ the operator 
\begin{equation} \label{eq:RIdefinition}
\mathcal{R}_{I(t)}: L^p(\Omega) \rightarrow L^p_{\sigma_t}(I(t)),\  v \mapsto (\int_{x\cdot \theta = s}v(x)\,dx)_{(s,t) \in I(t)}.	
\end{equation}
In order to deal appropriately with the undersampling we define 
\begin{equation}
	\Omega(t) = \{x \in \Omega~|~ \exists (\theta,s) \in I(t): s= x\cdot \theta~\}
\end{equation}
and assume that there exists a nonnegative Radon measure $\eta_t$ on $\Theta(t)$ and a bounded (uniformly in time) measurable function $\rho_t$ supported on $\Omega(t)$ such that
\begin{equation}
	\int_{I(t)} \int_{\{x \cdot \theta = s\}\cap \Omega} \psi(x,s,\theta) ~dx ~d\sigma_t= \int_{\Theta(t)} \int_\Omega	\psi(x,\theta\cdot x, \theta) \rho_t(x)~dx~d\eta_t
\end{equation}
for all integrable functions $\psi$. It is straight-forward to construct $\eta_t$ from $\sigma_t$ in the measurement scenarios outlined above. 
 
\begin{lemma} \label{RIlemma}
Let 
$$L=\sup_{(\theta,s) \in Z^2} \int_{\{x\cdot \theta = s\} \cap \Omega}\,dx < \infty.$$ Then $\mathcal{R}_{I(t)}$ is well-defined by \eqref{eq:RIdefinition} and a bounded linear operator for $p \in [1,\infty)$. The dual operator $\mathcal{R}_{I(t)}^* : L^{p_*}_{\sigma_t}(I(t)) \rightarrow L^{p_*}(\Omega) $ is given by
$$
\mathcal{R}_{I(t)}^* \varphi = \rho_t(x) ~\int_{\Theta(t)} \varphi(\theta,x\cdot \theta)~d\eta_t
$$ 
\end{lemma} 
\begin{proof}
First of all it is apparent that $\mathcal{R}_{I(t)}$ is well-defined on the dense subspace $C(\overline{\Omega})$. For continuous $v$ we have 
\begin{eqnarray*}\int_{I(t)} \left\vert \int_{x\cdot \theta = s}v(x)\,dx \right\vert^p ~d\sigma_t &\leq& L^{p_*} 
\int_{I(t)}  \int_{x\cdot \theta = s} \left\vert v(x) \right\vert^p \,dx~d\sigma_t \\
&=&  L^{p_*}  \int_{\Theta(t)} \int_\Omega	\left\vert v(x) \right\vert^p \rho_t(x)~dx~d\eta_t
 \\ &\leq&  L^{p_*}  \int_{\Theta(t)} ~d\eta_t \vert \rho_t\vert_\infty \vert v \vert_p^p.
 \end{eqnarray*}
Thus, $\mathcal{R}_{I(t)}$ is a bounded linear operator defined on a dense subspace and can be extended uniquely to a bounded linear operator on $L^p(\Omega)$.
Now consider the duality product
$$ \langle \varphi, \mathcal{R}_{I(t)}v \rangle  = \int_{I(t)}  \int_{x\cdot \theta = s} v(x) \varphi(\theta,s)\,dx~d\sigma_t = \int_{\Theta(t)} \int_\Omega v(x) \varphi(\theta,x\cdot \theta)~\rho_t(x)~dx~d\eta_t, $$
then an application of Fubini's theorem yields the above form of the dual operator.
\end{proof}

Then it follows directly that the operator $\mathcal{A}:L^p(\Omega\times [0,T])\rightarrow L^p(\cup_{t\in [0,T]} I(t)\times \{t\})$ is linear and bounded: 
\begin{proposition}
Let $[0,T]$ be a fixed and bounded time interval and $\Omega\subset\R^2$ bounded and let the assumptions of Lemma \ref{RIlemma} holds. Then for $p \in [1,\infty)$ the time-dependent Radon transform $\mathcal{A}:L^p(\Omega\times [0,T])\rightarrow L^p(\cup_{t\in [0,T]} I(t)\times \{t\})$, as defined in \eqref{eqn:RadonTransform}, is a well-defined bounded linear operator with dual operator
$$ (\mathcal{A}^* \varphi) (x,t) = (\mathcal{R}_{I(t)}^* \varphi(\cdot,t) ) (x).$$
\end{proposition}

Considering the inverse problem, since  we cannot measure the full sinogram in real life applications,  uniqueness of the solution $u$ in \eqref{eqn:measMod} is not guaranteed. Furthermore, the measurement $m$ is typically contaminated with noise and we need additional regularization to obtain a stable reconstruction. A well established approach is to search for a minimizer of a regularization functional, such as
\begin{equation}\label{eqn:RadonRecFunc}
J_{\text{rec}}(u):=\frac{1}{p}\|\mathcal{R}_{I(t)}u-m(t)\|_{p}^p + \alpha |u|_{BV},
\end{equation}
where $\alpha$ is a regularization parameter balancing the two parts. The first term in \eqref{eqn:RadonRecFunc} is the data fidelity term, which enforces that the sought-for attenuation function is close to the obtained measurement. The second term is the regularization term enforcing certain features of the reconstruction. In particular, we are interested in a sparse reconstruction with constant areas that are divided by sharp edges. For this purpose the so-called total variation is a common approach. In this study we consider the two choices $p\in\{1,2\}$. For $p=2$ we have the classical and well-studied $L^2$-$TV$ model used for tomographic imaging by  \cite{Herman2008,Jensen2011,Kolehmainen2006,Sidky2008,Tang2009,Tian2011} and many more, in contrary the $L^1$-$TV$ model is typically not used for X-ray tomography but tends to reduce streaking artifacts for undersampled data \cite{Sidky2012}. With \eqref{eqn:RadonRecFunc}, reminiscent of the well-known ROF-model \cite{Rudin1992}, we would obtain separate image reconstructions at different time steps, which is attractive from a computational point of view, but seems only applicable for a reasonable amount of measurements. In this study we will supplement the model by an appropriate time correlation of the image sequence $u$ and evaluate the performance of both models for extremely undersampled measurement data.

\subsection{Motion models}

Optical flow is one of the most common methods to estimate motion between consecutive images. Its performance is based on the assumption of brightness constancy, i.e. every pixel keeps its intensity over time even if it moves to another position within the image. Assuming a constant image intensity $u(x,t)$ along a trajectory $x(t)$ with $\frac{dx}{dt}=\boldsymbol{v}(x,t)$, we obtain
\begin{align}
	0 = \frac{du}{dt} =  \frac{\partial u}{\partial t} + \sum_{i=1}^n \frac{\partial u}{\partial x_i}\frac{dx_i}{dt} = u_t + \nabla u\cdot\boldsymbol{v}.
	\label{equation:opticalFlowConstraint}
\end{align}
The last equation is generally known as the \textbf{optical flow constraint} and $\boldsymbol{v}=(v^1,v^2)^T$ is the desired vector field. In spatial dimension two, \eqref{equation:opticalFlowConstraint} only states one equation per point for the two unknown components of $\boldsymbol{v}$ and, consequently, the problem is underdetermined. To overcome this, the optical flow formulation can be used as a data fidelity in a variational model together with an isotropic total variation term on each of the two flow components to ensure spatial regularity, i.e.
\begin{align}
J_{\text{flow}}(u,\boldsymbol{v})=\|u_t + \nabla u\cdot \boldsymbol{v}\|_1 + \beta\ |\boldsymbol{v}|_{BV}.
\label{equation:tvlonemodel}
\end{align}
In this model, the parameter $\beta>0$ regulates between both parts. In the optical flow setting the $L^1$ norm has been proven to be more robust with respect to outliers \cite{aubert1999computing}, which is an important characteristic especially in combination with noisy data from real applications. This model is nowadays one of the most popular models for optical flow, since it has shown success while tracking constant moving objects over time, see e.g. \cite{Zach2007}. The total variation regularization usually causes piecewise constant vector fields, which allow to distinguish a moving object from the background.

We mention that various modifications can be incorporated into our approach in a straightforward way. For out of plan motion it may be necessary to include additional source and sink terms to obtain
$$ u_t + \nabla u\cdot\boldsymbol{v} = S ,$$
with $S$ becoming another optimization model, ideally regularized by a sparsity prior in the variational model. For three-dimensional tomography reconstruction it is completely natural to replace the optical flow constraint by the continuity equation
$$ u_t + \nabla \cdot( u \boldsymbol{v} ) = 0. $$

\subsection{Motion corrected variational reconstruction }

From our point of view the problem of motion estimation is directly connected to the tomographic reconstruction problem because it requires accurate input images. In many motion estimation applications one first reconstructs the image sequence using \eqref{eqn:RadonRecFunc} and afterwards estimates the underlying vector fields using \eqref{equation:tvlonemodel}. In \cite{Dir15} it has been shown that a joint model that simultaneously recovers an image sequence and estimates motion offers a significant advantage towards subsequently applying both methods. In \cite{Frerking2016} the model proposed in \cite{Dir15} was applied to dynamic X-ray tomography, resulting in the following joint model:
\begin{align}\begin{split}
	\argmin_{u,\boldsymbol{v}}& \int_0^T \left( \frac{1}{p}\left\|(\mathcal{A}u)(\cdot,t)-m(\cdot,t) \right\|_p^p + \alpha |u(\cdot,t)|_{BV}^q + \beta | \boldsymbol{v}(\cdot,t)|_{BV}^r \right) dt,\\
	 \text{s.t. } & u_t+\nabla u\cdot\boldsymbol{v} = 0\end{split}
	\label{oldJointModel}
\end{align}
for $p\in\{1,2\}$ and $q,r \geq 1$. For both image sequence and vector field, the respective total variation is used as a regularizer and the classical optical flow formulation from \eqref{equation:opticalFlowConstraint} connects image sequence and vector field. From the perspective of image reconstruction the optical flow constraint acts as an additional temporal regularizer along the calculated motion field $\boldsymbol{v}$.

Appropriate weak-star compactness of sublevel sets and lower semicontinuity can be deduced from the arguments in \cite{burger2016variational}, where the minimization was carried out over the set 
\begin{equation}
	\mathcal{D}:= \{(u,\boldsymbol{v}) \in L^{\min{p,q}}(0,T;BV(\Omega)) \times L^r(0,T;BV(\Omega)~|~\| \boldsymbol{v}\|_{\infty} \le c_v < \infty, \vert \nabla \cdot \boldsymbol{v} \vert_{\mathcal{E}} \leq c_d ~\},
\end{equation}
where $\mathcal{E}$ is a Banach space continuously embedded into $L^m(0,T;L^k(\Omega))$,
with $m > q^*$ and $k >p$. Noticing that $BV(\Omega)$ is continuously embedded into $L^p(\Omega)$ for $p\leq 2$ in two spatial dimensions, the arguments of \cite{burger2016variational} can be directly applied for $p=2$ and the ones in \cite{Frerking2016} provide a similar proof for the case $p=1$. Thus, we obtain the following existence result for \eqref{oldJointModel}:
\begin{theorem}{(Existence of Minimizers)}\\
For $p\in\{1,2\}$, let $1<q,r$ and
\begin{align*}
J_{joint}(u,\boldsymbol{v}) =  \int_0^T \left( \frac{1}{p} \| (\mathcal{A}u)(\cdot,t) - m(t) \|_p^p +   \alpha | u(\cdot,t) |^q_{BV} + \beta| \boldsymbol{v}(\cdot,t) |^r_{BV} \right)\, dt.
\end{align*}
Furthermore, let $ \mathcal{R}_{I(t)}\boldsymbol{1} \neq 0$ for all $t \in [0,T]$.
Then, there exists a minimizer in the constraint set
\begin{align*}
\mathcal{S} = \left\lbrace (u,\boldsymbol{v}) \in \mathcal{D} \big|   \boldsymbol{v}\cdot\nabla u + \partial_t u = 0 \right\rbrace.
\end{align*}
\end{theorem}

We mention that the choice $q,r > 1$ has to be made in the analysis to avoid to deal with measures in time, in computational scenarios below it is however more efficient to set $q=r=1$.

\subsection{Uncertainty quantification for a probabilistic state-space model}

Estimating numerical errors in the problem \eqref{oldJointModel} is challenging due to the inherent nonlinearity of the problem. Below we provide a probabilistic perspective to error modeling by writing our dynamic system as a state-space model and applying Bayesian inference \cite{law2015data}. Here, we assume a time-space discretization which is specified later. Our state variable is $u^i$ at time-step $i$ and the flow field $\boldsymbol{v}^{i}$ corresponds to a latent variable, although it is naturally of equal interest. 
For convenience, let us write $\conu^k = (u^j)_{j=1}^k$ and $\conv^k = (\boldsymbol{v}^j)_{j=1}^k$ for the concatenation of the times series of state vectors until time $k$. Similarly, we write $\conf^k = (m^j)_{j=1}^k$. Notice that the flow field $\boldsymbol{v}$ is not estimated at the last time step and therefore $\conu^n$ and $\conv^{n-1}$ represent the full time series. Below, $\Pro(w)$ stands for the probability density function related to random variable $w$.

{\bf Prior.} Suppose our prior information regarding the initial state is given by the formal probability density
\begin{equation*}
	\Pro(u^1) \propto \exp\left(-\alpha \mathcal{G}_1(u^1)\right),
\end{equation*}
where $\mathcal{G}_1$ is the energy functional, e.g. the BV norm in Section \ref{sec:numerics}.
The equations governing optical flow yield the evolution model, which is given by
\begin{equation*}
	u^{i+1} = \mathcal{H}(u^i, \boldsymbol{v}^{i}) + \xi,
\end{equation*}
where $\mathcal{H}$ corresponds a suitable discretization of \eqref{equation:opticalFlowConstraint} specified later and the noise $\xi$ is distributed according to density $\Pro(\xi) \propto \exp(-\left\|\xi\right\|_1)$. Clearly, the conditional probability distribution of $u^{i+1}$ given both $u^i$ and $\boldsymbol{v}^{i}$ can be expressed by
\begin{equation*}
	\Pro(u^{i+1}	| u^i, \boldsymbol{v}^{i}) \propto
	\exp\left(-\gamma \left\|u^{i+1}-\mathcal{H}(u^i,\boldsymbol{v}^i)\right\|_1\right).
\end{equation*}
Here, we make the crucial assumption that $u^{i+1}$ and $\boldsymbol{v}^{i}$ is {\it a priori} {\it independent $\boldsymbol{v}^{i-1}$}. Notice carefully that this is not the case {\it a posteriori}. 
Moreover, we assume $\boldsymbol{v}^{i}$ is {\it a priori} independent of $u^i$ and therefore
\begin{equation*}
	\Pro(u^{i+1}, \boldsymbol{v}^{i}	| u^i)  = \Pro(u^{i+1}	| u^i, \boldsymbol{v}^{i}) \Pro(\boldsymbol{v}^{i}),
\end{equation*}
Now by assuming $\Pro(\boldsymbol{v}^{i}) \propto \exp(- \beta \mathcal{G}_2(\boldsymbol{v}^i))$ for some energy functional $\mathcal{G}_2$, it follows that
\begin{equation*}
	\Pro(u^{i+1}, \boldsymbol{v}^{i}	| u^i) \propto
	\exp\left(-\gamma \left\|u^{i+1}-\mathcal{H}(u^i,\boldsymbol{v}^i)\right\|_1  - \beta \mathcal{G}_2(\boldsymbol{v}^i)\right)
\end{equation*}
and the full prior model can be expressed recursively as
\begin{eqnarray*}
	\Pro(\conu^n, \conv^{n-1})
	& = & \Pro(u^n, \boldsymbol{v}^{n-1} | \conu^{n-1}, \conv^{n-2}) 
	\Pro(\conu^{n-1}, \conv^{n-2}) \\
	& = & \Pro(u^n, \boldsymbol{v}^{n-1} | u^{n-1})\Pro(\conu^{n-1}, \conv^{n-2}) \\
	& = & \prod_{i=1}^{n-1} \Pro(u^{i+1}, \boldsymbol{v}^{i} | u^{i}) \cdot \Pro(u^{1}).
\end{eqnarray*}

{\bf Likelihood.} Our observation of the system state is obtained via 
\begin{equation*}
	m^i = A^i u^i + \epsilon^i,
\end{equation*}
where $\epsilon^i$, $i=1,...,n$ are i.i.d. and $\Pro(\epsilon^i) \propto \exp(-\frac 1p \left\|\epsilon^i\right\|^p_p)$ is a random noise vector. Moreover, we assume virtual zero-observations at time steps $2,...,n$, i.e., 
we observe
\begin{equation}
	\label{eq:virtual_observations}
	g^i = u^i - \delta^i,
\end{equation}
where $\{\delta^i\}_{i=2}^n$ are i.i.d., $\Pro(\delta^i) \propto \exp(-\alpha \mathcal{G}_1(\delta^i))$ and assume $g^i = 0$ for all $i\geq 2$.
Notice that the virtual observations are not necessary for the probabilistic system to be well-defined. However, observations in \eqref{eq:virtual_observations} state that for the likely values of $u^i$ the quantity $\mathcal{G}_1(u^i)$ is small and, therefore, impose some additional regularity to the system.

Under these observations it follows that the likelihood density is of the form
\begin{equation*}
	\Pro(\conf^n \; |\; \conu^n) \propto \exp\left( -\sum_{i=1}^{n} \frac{1}{p}\left\|A^iu^i-m^i\right\|_p^p - \sum_{i=2}^n \alpha \mathcal{G}_1(u^i) \right).
\end{equation*}

{\bf Uncertainty quantification.} In this work we consider reconstruction methods based on {\it smoothing} \cite{law2015data}, i.e., we estimate all states simultaneously based on the full time-series data. In this case, the posterior distribution obtained from the state-space model is proportional to the product of the prior and likelihood
densities
\begin{equation}
	\label{eq:posterior}
	\Pro(\conu^n, \conv^{n-1} \;  | \; \conf^n) \propto
	\Pro(\conf^n | \conu^n) \Pro(\conu^n, \conv^{n-1}) 
	= \exp\left(-J(\conu^n, \conv^{n-1}; \conf^n)\right),
\end{equation}
where $J$ is the functional given by
\begin{multline}
	J(\conu^n, \conv^n; \conf^n) \\
	= \sum_{i=1}^{n}\left\{ \frac{1}{p}\left\|A^iu^i-m^i\right\|_p^p + \alpha \mathcal{G}_1(u^i)\right\} + \sum_{i=1}^{n-1}\left\{\gamma \left\|u^{i+1} - \mathcal{H}(u^i, \boldsymbol{v}^i)\right\|_1 + \beta \mathcal{G}_2(\boldsymbol{v}^i)\right\}
	\label{oldJointModelTimeDiscrete},
\end{multline}
Our numerical work below on estimating the maximum point of the posterior distribution, i.e. minimizer of $J$, corresponds to the {\it weak constraint 4DVAR} method \cite{apte2008data}.
Sampling the full posterior distribution requires high computational effort and is not explored in this paper.

Notice that the state-space model also enables {\it filtering} methods, where one is concerned by the online estimation of $\Pro(u^{i+1},\boldsymbol{v}^i | \conf^{i+1})$ as the observational data is accumulated.  
The update from $\Pro(u^{i},\boldsymbol{v}^{i-1} | \conf^{i})$ to $\Pro(u^{i+1},\boldsymbol{v}^i | \conf^{i+1})$ is obtained via a prediction step
\begin{equation*}
	\Pro(u^{i},\boldsymbol{v}^{i-1} | \conf^{i}) \mapsto \Pro(u^{i+1},\boldsymbol{v}^i | \conf^{i})
	= \int \int \Pro(u^{i+1},\boldsymbol{v}^i | u^{i},\boldsymbol{v}^{i-1})
	\Pro(u^{i},\boldsymbol{v}^{i-1} | \conf^{i}) du^i d\boldsymbol{v}^{i-1}
\end{equation*}
and an analysis step
\begin{equation*}
	\Pro(u^{i+1},\boldsymbol{v}^i | \conf^{i}) \mapsto
	\Pro(u^{i+1},\boldsymbol{v}^i | \conf^{i+1}) = 
	\frac{\Pro(m^{i+1} | u^{i+1}, \boldsymbol{v}^i) \Pro(u^{i+1}, \boldsymbol{v}^i | \conf^{i})}{\Pro(m^{i+1} | \conf^{i})}.
\end{equation*}
These aspects of the stochastic dynamics in optical flow models are studied further in subsequent work.

\section{Numerical implementation}\label{sec:numerics}

To handle the numerical implementation of the joint model for motion corrected reconstruction, we first need to formulate a discrete version. Here we especially focus on the structure of the time-dependent Radon operator respectively the matrix representing its discretization, in order to obtain efficient schemes. Afterwards, we split the variational model into two subproblems and introduce an alternating approach to solve it. In order to minimize the single subproblems, we employ iteration schemes based on established primal-dual methods.

\subsection{Discretization}
In practical applications we cannot measure infinitely many line integrals continuously in time and hence we need to assume a discretization in space and time to model the measurement process properly. We aim at representing the inverse problem of recovering the attenuation as a simple matrix-vector equation 
\begin{equation}\label{eqn:RadonLinearMatrixVec}
A\vecu=\vecm,
\end{equation}
with the matrix $A$ representing the discretized Radon transform dependent on time, $\vecu$ being the attenuation coefficient in each pixel. The measurement $\vecm$ is taken during a fixed time period $[0,T]$ as discussed in Section \ref{sec:RadonTrafo}, and we only measure at certain time instances. We denote the number of measured time instances by $N_t$, such that each measurement point in time is given by $t_k=\frac{kT}{N_t}$.

Let us consider one fixed point in time $t_k$. Then for the discretization in space we divide the domain into $N_x$ pixels such that the attenuation is modeled by a vector $\vecu\in\R^{N_x}$. Furthermore, we have a finite set of $N_\ell$ lines $\ell_1,\dots,\ell_{N_\ell}$ for which we can measure the attenuation. The total amount of lines depends on the projection angles in each step and is typically a multiple of the sensor resolution. In case we have only one projection per time step, then $N_{\ell}$ coincides with the sensor resolution. 
The discrete measurement is then given by
\begin{equation}\label{eqn:RadonDiscreteMeas}
\vecm_i=\sum_{j=1}^{N_x} a_{i,j}\vecu_j, \hspace{0.25cm} i=1,\dots,N_\ell,
\end{equation}
where $a_{i,j}$ is the length of the line $\ell_i$ in the $j$th pixel. The measurement matrix at time $t_k$ denoted by $A^k\in\R^{N_\ell\times N_x}$ is then composed of the coefficients in \eqref{eqn:RadonDiscreteMeas}.
The matrix for the projection of all time steps is simply given as the block diagonal matrix
\begin{align}
	A=\left(
	\begin{array}{cccc}
	A^1 & & & 0 \\ 
	& A^2 & & \\
	& & \ddots &  \\
	0 & &  & A^{N_t}
	\end{array} \right).\label{matrixBlockStructure}
\end{align}

Concerning the entire problem \eqref{oldJointModel}, the discretization of the time-interval $[0,T]$ into $N_t$ steps leads to $N_t$ images $u^1,\ldots,u^{N_t}$ resp. $N_t-1$ vector fields $\boldsymbol{v}^1,\ldots,\boldsymbol{v}^{N_t-1}$ between subsequent frames. Using the previously derived discrete Radon matrix, the time-discrete counterpart of \eqref{oldJointModel} is hence given by
\begin{equation}
	\argmin_{\substack{u = (u^1,\ldots, u^{N_t})\\\boldsymbol{v}=(\boldsymbol{v}^1,\ldots\boldsymbol{v}^{N_t-1})}} \sum_{i=1}^{N_t} \frac{1}{p}\left\|A^iu^i-m^i\right\|_p^p + \alpha \left\|\nabla u^i\right\|_{2,1} + \sum_{i=1}^{N_t-1}\gamma \left\|u^{i+1} - u^i +\nabla u^i\cdot\boldsymbol{v}^i\right\|_1 + \beta\sum_{j=1}^2\left\|\nabla v^{i,j}\right\|_{2,1}
	\label{oldJointModelTimeDiscrete},
\end{equation}
where $\|\nabla u\|_{2,1}$ denotes the isotropic total variation, which takes the point wise Euclidean norm of $(u_x,u_y)\in\R^{N_x}\times\R^{N_x}$ and afterwards sums up the resulting vector in $\R^{N_x}$. 
\subsection{Minimization}
Despite showing the existence of a minimizer, its calculation is numerically challenging. Problems arise from the non-convexity and non-linearity of the optical flow term, the non-differentiability of the $L^1$--norm and finally several linear operators acting on $u$ and $\boldsymbol{v}$. To address these issues, the joint model \eqref{oldJointModelTimeDiscrete} can be transformed into a two-step method
\begin{align}
u_{l+1} = \argmin_{u} & \sum_{i=1}^{N_t} \frac{1}{p}\left\|A^iu^i-m^i\right\|_p^p  + \alpha \left\|\nabla u^i\right\|_{2,1} + \gamma  \sum_{i=1}^{N_t-1}\left\|u^{i+1} - u^i +\nabla u^i \cdot \boldsymbol{v}^i_l\right\|_1 \label{eq:simultanTVTVunconstrainedSubU},\\
\boldsymbol{v}_{l+1} = \argmin_{\boldsymbol{v}} & \sum_{i=1}^{N_t-1}  \left\|u^{i+1}_{l+1} - u^i_{l+1}+\nabla u^i_{l+1}\cdot \boldsymbol{v}^i\right\|_1 + \frac{\beta}{\gamma}\sum_{j=1}^2\left\|\nabla v^{i,j}\right\|_{2,1}  \label{eq:simultanTVTVunconstrainedSubV}
\end{align}
that alternatingly solves a problem for the image sequence $u$ using information from $\eqref{eq:simultanTVTVunconstrainedSubV}$ and a problem for the flow sequence $\boldsymbol{v}$ using information from \eqref{eq:simultanTVTVunconstrainedSubU}. Using block diagonal operators, see \eqref{matrixBlockStructure}, equation \eqref{eq:simultanTVTVunconstrainedSubU} and \eqref{eq:simultanTVTVunconstrainedSubV} can be simplified into problems of the form
\begin{align}
u_{l+1} = \argmin_{u} & \frac{1}{p}\left\|Au-m\right\|_p^p  + \alpha \left\|\nabla u\right\|_{2,1} + \gamma  \left\|Tu\right\|_1 \label{eq:simultanTVTVunconstrainedSubUSimple},\\
\boldsymbol{v}_{l+1} = \argmin_{\boldsymbol{v}} & \left\|\hat{T}\boldsymbol{v} - b\right\|_1 + \frac{\beta}{\gamma}\sum_{j=1}^2\left\|\nabla v^{j}\right\|_{2,1}  \label{eq:simultanTVTVunconstrainedSubVSimple},
\end{align}
where the operator $T$ depends on $\boldsymbol{v}_{l}$ and $\hat{T}$ depends on $u_{l+1}$. The terms $u^{i+1}_{l+1} - u^i_{l+1}$ from \eqref{oldJointModelTimeDiscrete} can be combined to the vector $b$ in \eqref{eq:simultanTVTVunconstrainedSubVSimple}.
Due to this transformation, each of these subproblems is linear and convex but still non-differentiable. Moreover, the alternating scheme tends to end up in local minima rather than in convergence to the solution of \eqref{oldJointModel}.

In order to minimize the single subproblems for reconstruction and for motion estimation, we use a primal-dual approach, which was introduced in \cite{chambolle2011first}. Therefore, we rewrite the given problems as saddle point formulations of the form
\begin{align*}
\min_{u\in\mathcal{X}}\max_{p\in\mathcal{Y}}  \left\langle Ku,p \right\rangle -F^*(p)
\end{align*}
and subsequently apply the iteration scheme proposed in \cite{chambolle2011first}. In what follows we address the problems for $u$ and $\boldsymbol{v}$ in detail. The pseudocode in Algorithm \ref{algorithmMain} then gives a sketch of the alternating minimization strategy.

{\bf Reconstruction.} Within this part we restrict ourselves to the case of an $L^1$ data fidelity where an extension to an $L^2$ term is straight forward. Since all terms of \eqref{eq:simultanTVTVunconstrainedSubU} contain operators, we dualize each term and obtain the saddle-point problem
%
\begin{align*}
\min_{u}\max_{p_1,p_2,p_3} \sum_{i=1}^{3} \langle K_iu,p_i\rangle - F^*_i(p_i),
\end{align*}
with 
\begin{align*}
K_1u = Au, & \quad F^*_1(p_1) = \delta_{\{ \bar{p}_1 : \|\bar{p}_1\|_\infty \leq 1 \}}(p_1) + \langle p_1,f\rangle,\\
K_2u = \nabla u, & \quad F^*_2(p_2) = \delta_{\{ \bar{p}_2 : \|\bar{p}_2\|_{2,\infty} \leq \alpha \}}(p_2),\\
K_3u = Tu, & \quad F^*_3(p_3) = \delta_{\{ \bar{p}_3 : \|\bar{p}_3\|_\infty \leq \gamma \}}(p_3).\\
\end{align*}
Here $\delta_\mathcal{C}$ denotes an indicator function on the set $\mathcal{C}$ defined as
\begin{align*}
	\delta_\mathcal{C}(u) = \begin{cases}
	0 & \text{if } u\in\mathcal{C}\\
	\infty & \text{else}
	\end{cases}.
\end{align*}
An application of the primal-dual method to the above problem yields the following iteration scheme:
\begin{align}\begin{split}\label{eq:iterativeSchemeU}
p_1^{k+1} &= \pi_{\{ \bar{p_1} : \|\bar{p_1}\|_{\infty} \leq 1 \}}(p^k_1+\sigma_u A\bar{u}^k-\sigma_u f),\\
p_2^{k+1} &= \pi_{\{ \bar{p_2} : \|\bar{p_2}\|_{2,\infty} \leq \alpha \}}(p^k_2+\sigma_u \nabla\bar{u}^k),\\
p_3^{k+1} &= \pi_{\{ \bar{p_3} : \|\bar{p_3}\|_{\infty} \leq \gamma \}}(p^k_3+\sigma_u T\bar{u}^k),\\
u^{k+1} &= u^k-\tau_u (A^Tp_1^{k+1} + \nabla^Tp_2^{k+1} + T^Tp_3^{k+1}),\\
\bar{u}^{k+1} &= 2u^{k+1}-u^k,
\end{split}\end{align}
where $\pi_\mathcal{C}$ denotes a projection of the input argument onto the set $\mathcal{C}$ and $\sigma_u,\tau_u$ are valid stepsizes as explained in \cite{chambolle2011first}.


{\bf Motion Estimation.}
The saddle point formulation for the motion estimation problem can be derived in analogy to the reconstruction problem as follows

\begin{align*}
\min_{\boldsymbol{v}=(v_1,v_2)}\max_{q_1,q_2,q_3} \sum_{i=1}^{3} \langle K_i\boldsymbol{v},q_i\rangle - F^*_i(q_i),
\end{align*}
where 
\begin{align*}
K_1\boldsymbol{v} = \hat{T}\boldsymbol{v}, & \quad F^*_1(q_1) = \delta_{\{ \bar{q}_1 : \|\bar{q}_1\|_{\infty} \leq 1 \}}(q_1)  - \langle q_1,b\rangle,\\
K_2\boldsymbol{v} = \nabla v_1, & \quad F^*_2(q_2) =  \delta_{\{ \bar{q}_2 : \|\bar{q}_2\|_{2,\infty} \leq \frac{\beta}{\gamma} \}}(q_2),\\
K_3\boldsymbol{v} = \nabla v_2, & \quad F^*_3(q_3) = \delta_{\{ \bar{q}_3 : \|\bar{q}_3\|_{2,\infty} \leq \frac{\beta}{\gamma} \}}(q_3).\\
\end{align*}

Similar to the reconstruction part, we apply the iteration scheme proposed in \cite{chambolle2011first} to the given problem and end up with
\begin{align}\begin{split}\label{eq:iterativeSchemeV}
q_1^{k+1} &= \pi_{\{ \bar{q}_1 : \|\bar{q}_1\|_{\infty} \leq 1 \}}(q^k_1+\sigma_v \hat{T}\bar{\boldsymbol{v}}^k-\sigma_v b),\\
q_2^{k+1} &= \pi_{\{ \bar{q}_2 : \|\bar{q}_2\|_{2,\infty} \leq \frac{\beta}{\gamma} \}}(q^k_2+\sigma_v \nabla\bar{v}^{k}_1),\\
q_3^{k+1} &= \pi_{\{ \bar{q}_3 : \|\bar{q}_3\|_{2,\infty} \leq \frac{\beta}{\gamma} \}}(q^k_3+\sigma_v \nabla\bar{v}^{k}_2),\\
\boldsymbol{v}^{k+1} &= \boldsymbol{v}^k-\tau_v \left(\hat{T}^{T}q_1^{k+1} + \begin{pmatrix}\nabla^T&0\\0&\nabla^T\end{pmatrix} \begin{pmatrix}q_2^{k+1}\\q_3^{k+1}\end{pmatrix}\right),\\
\bar{\boldsymbol{v}}^{k+1} &= 2\boldsymbol{v}^{k+1}-\boldsymbol{v}^k.
\end{split}\end{align}

To handle large displacements, the optical flow calculation is incorporated into a coarse-to-fine pyramid with intermediate warping steps. We refer to \cite{dirks2016joint} for details.

\begin{algorithm}
	\caption{Joint Reconstruction: Main Algorithm}
	\label{algorithmMain}
	\begin{algorithmic}
		\Statex
		\Function{mainFunction}{$m$}
		\Let{$u,\boldsymbol{v}$}{0}
		\State Initialize Radon operator $A$
		\While{$r_{main}>tol_{main}$}
		\Let{$u_{old}$}{$u$}
		\Let{$\boldsymbol{v}_{old}$}{$\boldsymbol{v}$}
		\State Update operator $T$ using $\boldsymbol{v}$
		\Let{$u$}{Run scheme \eqref{eq:iterativeSchemeU} until converged}
		\State Update operator $\hat{T}$ using $u$
		\Let{$\boldsymbol{v}$}{Run scheme \eqref{eq:iterativeSchemeV} until converged}
		\Let{$r_{main}$}{$\|u-u_{old}\| + \|\boldsymbol{v}-\boldsymbol{v}_{old}\|$}
		\EndWhile
		\State \Return{$(u,\boldsymbol{v})$}
		\EndFunction
	\end{algorithmic}
\end{algorithm}


\section{Experiments}\label{sec:experiments}
For the evaluation of the proposed motion estimation and reconstruction we consider two experiments in this section. For a qualitative evaluation we consider a simulated data set of a moving ball. Since the ground truth is known we can explicitly evaluate the performance of the $L^1\text{-}TV$ and $L^2\text{-}TV$ model as reconstruction functional \eqref{eqn:RadonRecFunc} by evaluating the reconstruction errors. Based on the knowledge obtained in the simulated experiments we then apply the reconstruction algorithm to real measurement data from the $\mu$CT lab at the University of Helsinki.

A further aim is to compare the reconstruction quality of the $L^1$ and $L^2$ fidelity terms for extremely undersampled dynamic data by means of their reconstruction error. Two of the most common error measures are the $l_1$ distance and the $l_2$ distance between the reconstruction and the phantom. Thus, we also use both measures to compare the results. However, a disadvantage of their utilization is that they are not neutral with respect to the chosen norms of the model. Naturally, the error in $l_1$ is smaller for an $L^1$ data fidelity term, and the error in $l_2$ is smaller for the $L^2$ model. Thus, it is necessary to use an unbiased technique. For this purpose, the method of our choice is the Structural Similarity (SSIM) index (cf \cite{brunet2012mathematical}). The SSIM index evaluates an image by comparing the structures of two images with a perception-based model. It includes illumination as well as different contrasts in an image. For two images $u_1$ and $u_2$ the SSIM index is given by

\begin{align*}
\widehat{SSIM}(u_1,u_2)=\frac{(2\mu_1\mu_2+c_1)(2\sigma_{12}+c_2)}{(\mu_1^2+\mu_2^2+c_1)(\sigma_1^2+\sigma_2^2+c_2)},
\end{align*}
where $\mu_1$ and $\mu_2$ are the averages of $u_1$ and $u_2$, $\sigma_1$ and $\sigma_1$ are the variances of $u_1$ and $u_2$, $\sigma_{12}$ is the covariance of $u_1$ and $u_2$, and $c_1$ as well as $c_2$ are constants to prevent a division by zero. To compare the structures of two entire image sequences, we add up the SSIM index for every time step and calculate the average, i.e.
\begin{align}\label{eq:ssim}
SSIM (u_1,u_2) = \frac{1}{T+1} \sum_{t=0}^T \widehat{SSIM}\left( u_1(x,t),u_2(x,t) \right).
\end{align}
In contrast to the errors in $l_1$ and in $l_2$, the SSIM index is close to one for images that are similar to the reference image.\\

\subsection{Software experiment: Pinball}
The synthetic Pinball data set consists of a two dimensional image of $42 \times 42$ pixels. The image is recorded at 30 consecutive time steps. Throughout this period of time, a uniform and rigid ball is moving from the left side of the image frame to the right side. During the whole time, the ball resides in a stationary ellipse of medium intensity. \Cref{fig:ballGT} shows the ground truth images at a selection of time steps. To avoid inverse crime, we first compute the sinogram from a high resolution image, add $1\%$ Gaussian noise, and finally downscale it to the size of the corresponding $42 \times 42$ phantom. The noise level has been chosen to be reasonably low for accurate measurements, such that the reconstructed features are mainly depending on sampled data and chosen angles.  The regularization parameters were chosen such that the $\ell^1$ error is minimized for the $L^1$-$TV$ model, $\ell^2$ for $L^2$-$TV$ respectively. That is $\alpha = 0.1,\ \beta = 0.2,\ \gamma = 0.5$ for $p=1$; and $\alpha = 0.05,\ \beta = 0.2,\ \gamma = 8$ for $p=2$. A full space-time reconstruction takes a few minutes on a modern CPU.

\begin{figure}[t!]
\centering
\begin{picture}(430,100)
\put(0,0){\includegraphics[width=.24\textwidth]{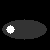}}
\put(110,0){\includegraphics[width=.24\textwidth]{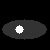}}
\put(220,0){\includegraphics[width=.24\textwidth]{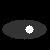}}
\put(330,0){\includegraphics[width=.24\textwidth]{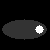}}
\end{picture}	\caption{Ground truth of Pinball data set. From left to right: time steps 1, 10, 20, 30.}
	\label{fig:ballGT}
\end{figure}

\begin{figure}[t!]
\centering
\begin{picture}(430,430)
\put(0,330){\includegraphics[width=.24\textwidth]{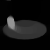}}
\put(110,330){\includegraphics[width=.24\textwidth]{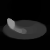}}
\put(220,330){\includegraphics[width=.24\textwidth]{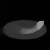}}
\put(330,330){\includegraphics[width=.24\textwidth]{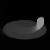}}
\put(0,220){\includegraphics[width=.24\textwidth]{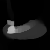}}
\put(110,220){\includegraphics[width=.24\textwidth]{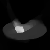}}
\put(220,220){\includegraphics[width=.24\textwidth]{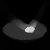}}
\put(330,220){\includegraphics[width=.24\textwidth]{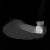}}
\put(0,110){\includegraphics[width=.24\textwidth]{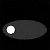}}
\put(110,110){\includegraphics[width=.24\textwidth]{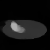}}
\put(220,110){\includegraphics[width=.24\textwidth]{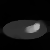}}
\put(330,110){\includegraphics[width=.24\textwidth]{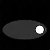}}
\put(0,0){\includegraphics[width=.24\textwidth]{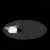}}
\put(110,0){\includegraphics[width=.24\textwidth]{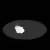}}
\put(220,0){\includegraphics[width=.24\textwidth]{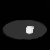}}
\put(330,0){\includegraphics[width=.24\textwidth]{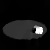}}
\end{picture}	\caption{Reconstruction result for Pinball data set calculated with $L^1$ data fidelity. From left to right: time steps 1, 10, 20, 30. From top to bottom: small angular increments, small angular increments with two angles, tracking, and randomized angles.}
	\label{fig:ballL1}
\end{figure}

\begin{figure}[t!]
\centering
\begin{picture}(430,430)
\put(0,330){\includegraphics[width=.24\textwidth]{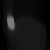}}
\put(110,330){\includegraphics[width=.24\textwidth]{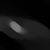}}
\put(220,330){\includegraphics[width=.24\textwidth]{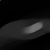}}
\put(330,330){\includegraphics[width=.24\textwidth]{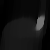}}

\put(0,220){\includegraphics[width=.24\textwidth]{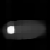}}
\put(110,220){\includegraphics[width=.24\textwidth]{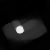}}
\put(220,220){\includegraphics[width=.24\textwidth]{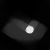}}
\put(330,220){\includegraphics[width=.24\textwidth]{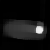}}

\put(0,110){\includegraphics[width=.24\textwidth]{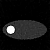}}
\put(110,110){\includegraphics[width=.24\textwidth]{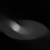}}
\put(220,110){\includegraphics[width=.24\textwidth]{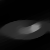}}
\put(330,110){\includegraphics[width=.24\textwidth]{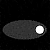}}

\put(0,0){\includegraphics[width=.24\textwidth]{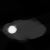}}
\put(110,0){\includegraphics[width=.24\textwidth]{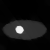}}
\put(220,0){\includegraphics[width=.24\textwidth]{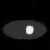}}
\put(330,0){\includegraphics[width=.24\textwidth]{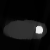}}
\end{picture}	\caption{Reconstruction result for Pinball data set calculated with $L^2$ data fidelity. From left to right: time steps 1, 10, 20, 30. From top to bottom: small angular increments, small angular increments with two angles, tracking, and randomized angles.}
	\label{fig:ballL2}
\end{figure}

The results of our computations can be seen in \Cref{fig:ballL1} for the $L^1$ data fidelity term and in \Cref{fig:ballL2} for $L^2$. Each figure illustrates the four measurement settings mentioned in Section \ref{sec:intro}. 
The results in the top rows are computed with small angular increments, i.e. single consecutive angles. As previously discussed, the information from a single angle is not sufficient to obtain a reasonable reconstruction and hence different measurement protocols need to be considered.
The second rows show results calculated assuming small angular increments with multiple angles, in this case two angles with a 90$^\circ$ offset.
The results for the tracking approach are presented in the third rows, i.e. we have a full CT scan of 60 angles available for the first and last time step.
The bottom rows present the results for one single randomized angle in each time step.

In both examples we can clearly see that the reconstructions from one incremental angle per time step can not produce satisfactory results. Even though the ellipses are rather well reconstructed, in case of the $L^1$ fidelty term sharp and for the $L^2$ fidelty term blurred, but the position and the shape of the balls are incorrect. The balls seems to follow a wave-like shape, which depends on the measurement angles. 
By increasing the angles per time step in the second measurement setup to two angles, the shape and the position of the balls are already significantly better reconstructed than in the results with only a single angle. However, for both models there are two tail-like artifacts in direction of the projections. Between those tails, i.e. in the upper part of the ellipse, edges are not properly reconstructed.
If one considers the tracking approach, both initial and end reconstructions are of course well reconstructed. However, in the intermediate time steps the position of the balls can not be reconstructed correctly. Though, the results are considerably more accurate than in the approach without any {\it a priori} information. 
Finally, for a randomized measurement protocol with a single angle, the balls are correctly located. For the $L^1$ fidelity term the shape is close to a square, whereas the $L^2$ fidelity term nicely reproduces a round appearance. The ellipse is well reconstructed for both fidelity terms. In general we note that the $L^1$ fidelty term produces sharper results, especially for the stationary ellipse and the $L^2$ fidelty term has a tendency to blur the background, but keeping the shape of the moving object closer to the original.\\

\begin{table}[h!]
	\centering
	\begin{tabular}{lcccccc}
		\hline \noalign{\smallskip}
		&\multicolumn{3}{c}{\textbf{$L^1$ data fidelity}} & \multicolumn{3}{c}{\textbf{$L^2$ data fidelity}}\\
		&$\ell_1$ error & $\ell_2$ error &SSIM & $\ell_1$ error & $\ell_2$ error &SSIM\\
		\noalign{\smallskip}\hline\noalign{\smallskip}
		small incr. ($1$ angle) & 0.4744 & 0.6485 & 0.7498 & 0.8897 & 0.6962 & 0.4310\\
		small incr. ($2$ angles) & 0.3828 & 0.4166 & 0.7275 & 0.3329 & 0.2954 & 0.7208\\
		tracking & 0.3131 & 0.5177 & 0.8240 & 0.4789 & 0.5042 & 0.6321\\
		$1$ randomized angle & 0.1978 & 0.3310 & 0.8502 & 0.2223 & 0.2586 & 0.8006\\
		\noalign{\smallskip}\hline\noalign{\smallskip}
	\end{tabular}
	\caption{Calculated relative $\ell_1$, $\ell_2$ errors, and averaged SSIM for Pinball reconstructions.}
	\label{tab:ballerrors}
\end{table}

\begin{figure}[h!]
\centering
\begin{picture}(430,220)
\put(0,110){\includegraphics[width=.24\textwidth]{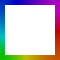}}
\put(110,110){\includegraphics[width=.24\textwidth]{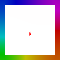}}
\put(220,110){\includegraphics[width=.24\textwidth]{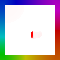}}
\put(330,110){\includegraphics[width=.24\textwidth]{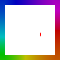}}
\put(0,0){\includegraphics[width=.24\textwidth]{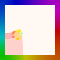}}
\put(110,0){\includegraphics[width=.24\textwidth]{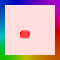}}
\put(220,0){\includegraphics[width=.24\textwidth]{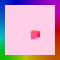}}
\put(330,0){\includegraphics[width=.24\textwidth]{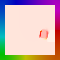}}
	\end{picture} \caption{Motion estimation results for Pinball data set calculated from randomized angles. From left to right: time steps 1, 10, 20, 29. Top: $L^1$ data fidelity, bottom: $L^2$ data fidelity.}
	\label{fig:flowrand}
\end{figure}

\noindent Additionally, to accompany the visual inspection, we computed the relative errors in $\ell_1$ and in $\ell_2$, as well as the SSIM index (see \eqref{eq:ssim}), displayed in \Cref{tab:ballerrors}. Comparing the performances of the $L^1$ and of the $L^2$ data fidelity for each experimental setting reveals that as expected the error in $\ell_1$ is always smaller for the results calculated with the $L^1$ data fidelity, whereas the $\ell_2$ error is smaller for the results of the $L^2$ data fidelity. However, the SSIM index, which we expect to be more neutral with respect to the chosen norm, indicates that the $L^1$ norm outperforms the $L^2$ norm for every single approach. This is a consequence of the distinctly better reconstruction of the ellipses. Concerning different experimental settings, the approach considering randomized angles achieves the best results by far. For both data fidelities all error measures indicate that this approach yields the best outcome. Consistent with the visual inspection, the approach with small angular increments with single angles can be considered as the worst performing approach.

To conclude this chapter, we present in \Cref{fig:flowrand} the flow fields corresponding to the approach with randomized angles and for both fidelity terms. The flow fields generally estimate the motion in the correct direction. Nevertheless, there are obvious differences between the flow fields estimated with the $L^1$ and the $L^2$ data fidelity. For the $L^2$ data fidelity the entire movement of the ball is visible in the flow field. In contrast to the $L^1$ data fidelity, here the model only recognizes motion at the edges of the ball. For both data fidelity terms the flow fields for the first time steps are not correctly estimated. The reason for this is the missing possibility to use {\it a priori} information for the estimation from previous time steps that can be used as initialization.\\

\subsection{Hardware experiment: Rolling Stones} 

\begin{figure}[h!]
\centering
\begin{picture}(450,280)
\put(0,145){\includegraphics[width=.32\textwidth]{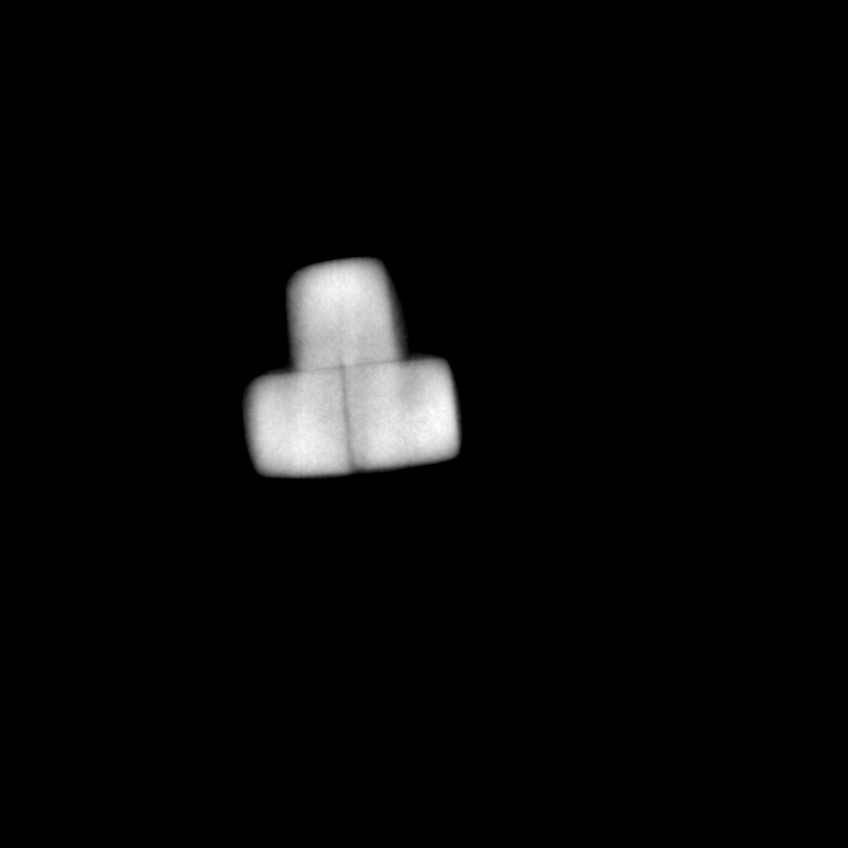}}
\put(150,145){\includegraphics[width=.32\textwidth]{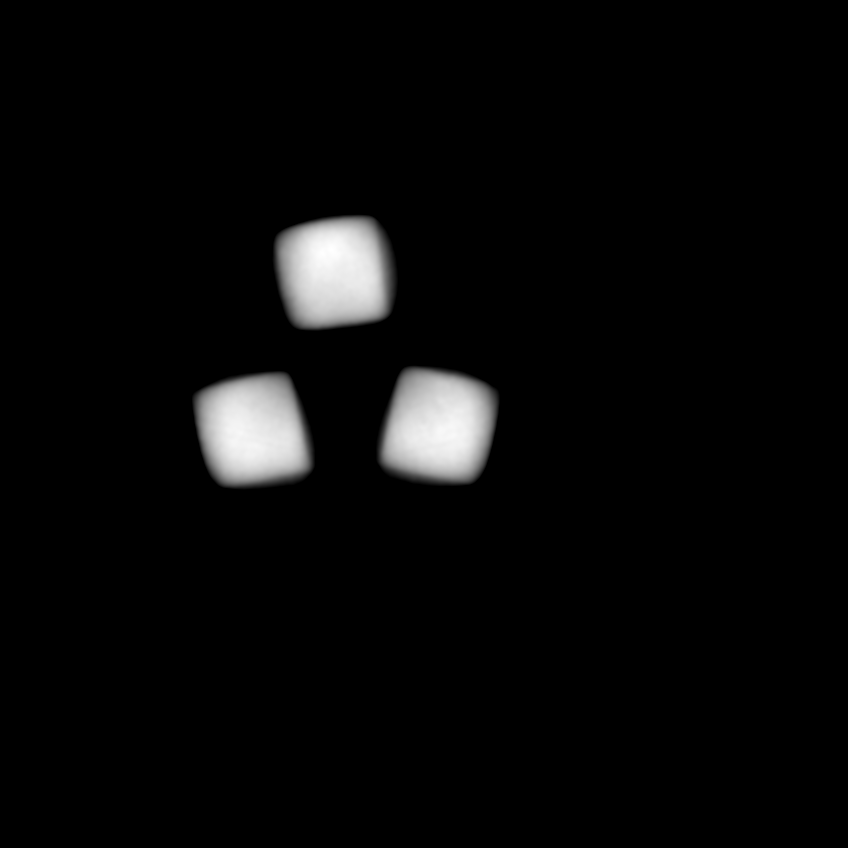}}
\put(300,145){\includegraphics[width=.32\textwidth]{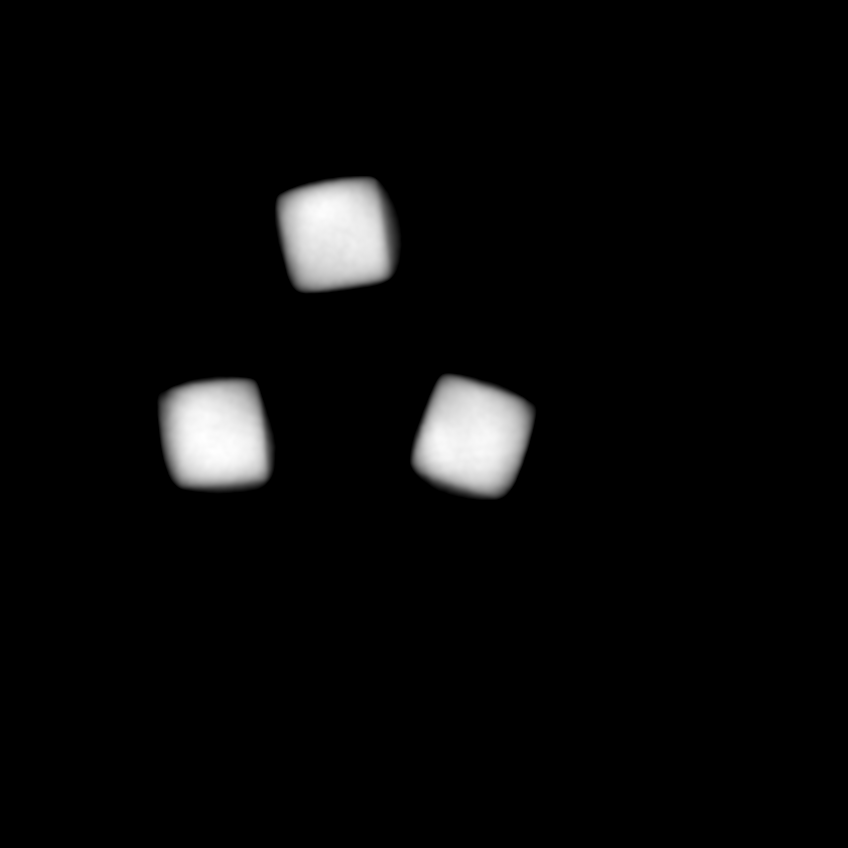}}
\put(0,-5){\includegraphics[width=.32\textwidth]{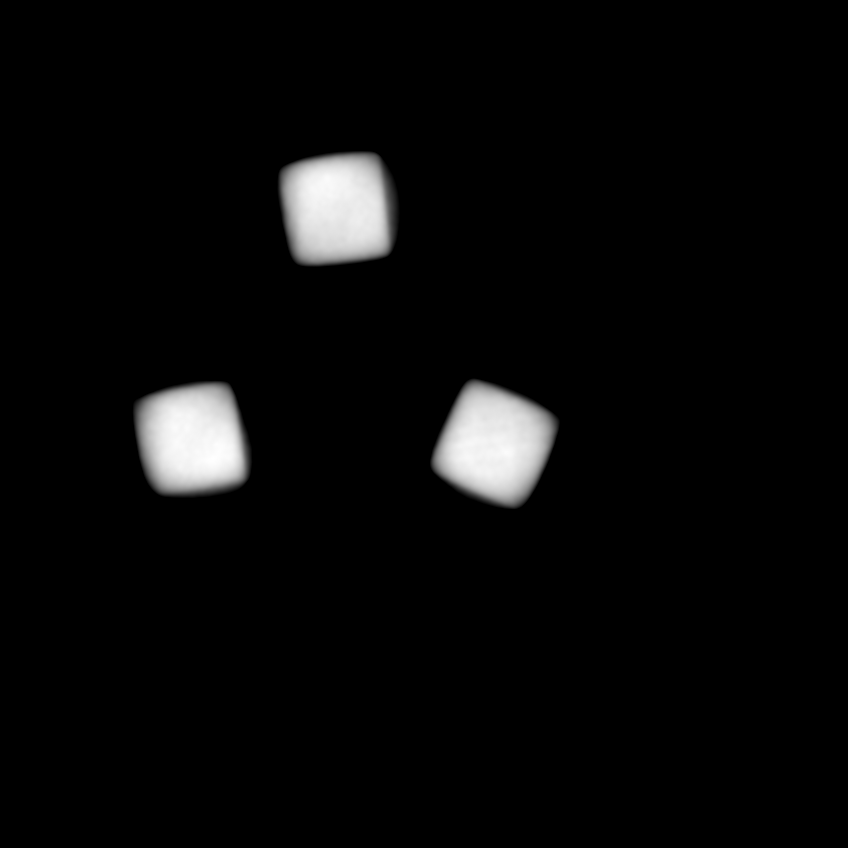}}
\put(150,-5){\includegraphics[width=.32\textwidth]{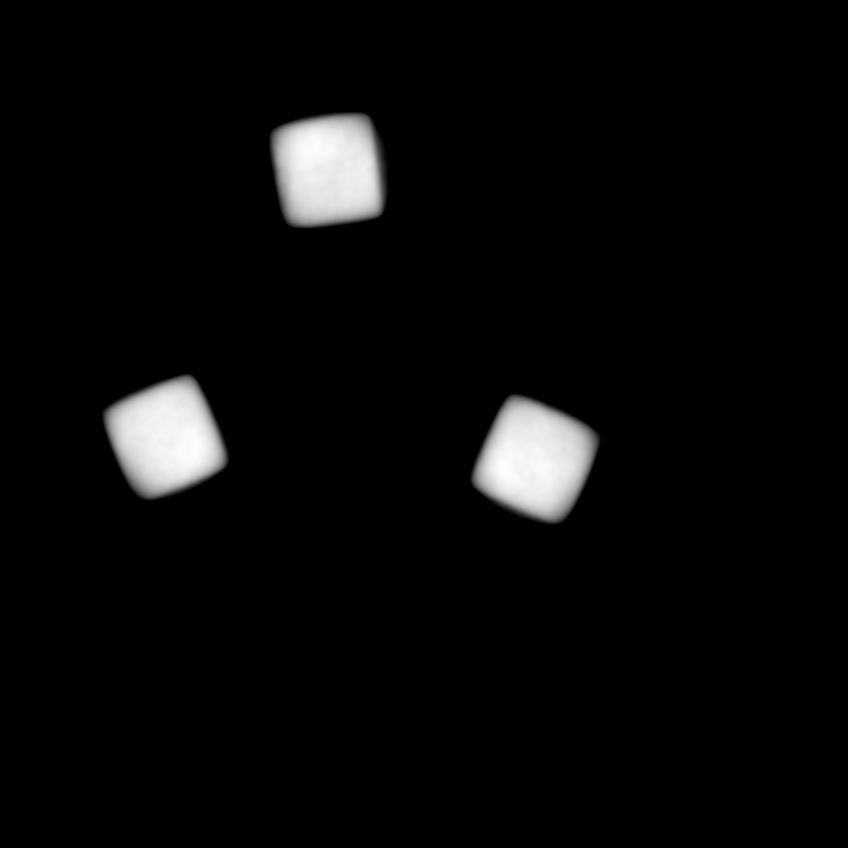}}
\put(300,-5){\includegraphics[width=.32\textwidth]{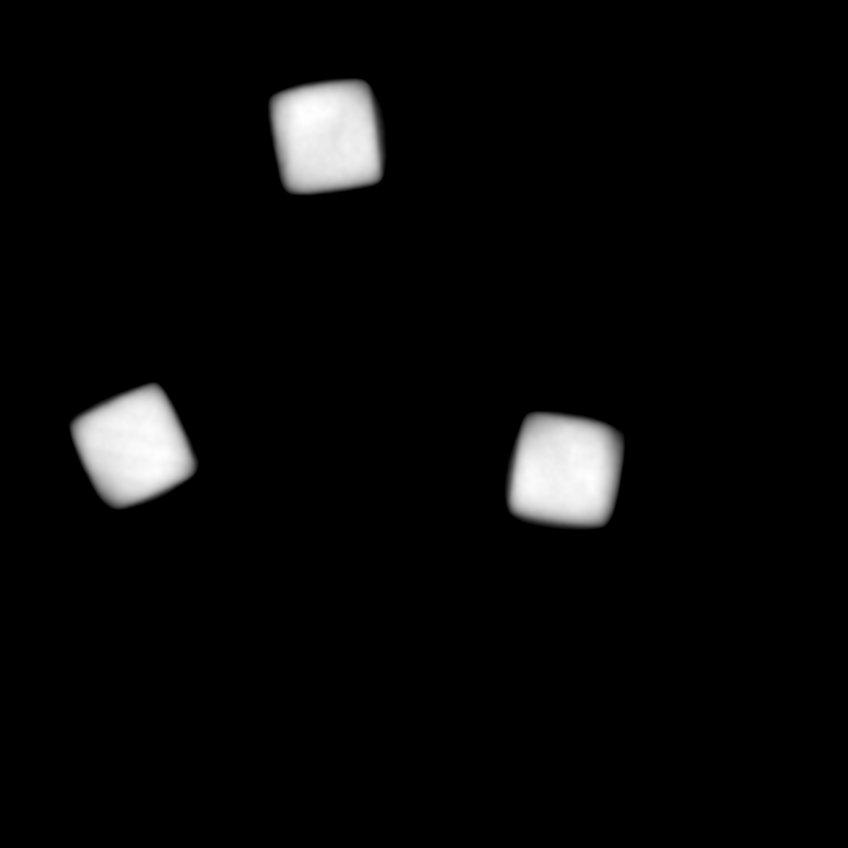}}
\end{picture}	\caption{Full-angle and high-resolution reconstruction of the Rolling Stones data. From top left to bottom right: time steps 1, 7, 13, 18, 25, 30.}
	\label{fig:stonesGT}
\end{figure}

The Rolling Stones data set consists of images of size $42 \times 42$ pixels, measured at 30 consecutive time steps. Even though we aim at being able to recover continuous movement measured with an extremely limited amount of angles, the measurements were actually recorded from 60 equally distributed angles in a stop-and-go approach. This has the advantage that we are able to use the exact same data set for different arrangements of angles as well as having a reference reconstruction from 60 angles as ground truth. \Cref{fig:stonesGT} shows the reconstruction from 60 angles for a representative selection of time steps, computed by a simple smoothed and therefore differentiable $L^2$-$TV$ variant especially suitable for large-scale data, a detailed description for the used procedure can be found in \cite{hamalainen2014total}.

The Rolling Stones data depicts three ceramic stones of approximately 25 mm$^2$, which are initially located next to each other in the center of the domain. In each time step the stones move further apart from each other to the boundary of the imaging domain. Since the stones were moved manually during the measurements, the vector and direction of movement differs for every stone and every time step.  We have seen in the software experiments that the $L^1$ data fidelity works best for reconstructing this kind of data and hence we restrict ourselves in the following to the $L^1$-$TV$ model. In \Cref{fig:sinoStones} we present the sinograms that are used in the following for the joint image reconstruction and motion estimation. \\

\begin{figure}[t!]
\centering
\begin{picture}(400,345)
\put(0,175){\includegraphics[height=.29\textheight]{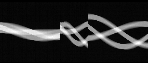}}
\put(0,-5){\includegraphics[height=.29\textheight]{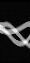}}
\put(125,-5){\includegraphics[height=.29\textheight]{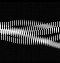}}
\put(330,-5){\includegraphics[height=.29\textheight]{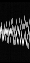}}
\end{picture}
\caption{Measured sinograms of the Rolling Stones data set. Top: Tracking approach with 60 angles at the first and last time step. Bottom from left to right: small angular increments with one angle, small angular increments with two angles, one randomized angle per time step}
\label{fig:sinoStones}
\end{figure}

We omit the reconstructions for small angular increments with one angle here, since the reconstructions were simply not satisfactory. On the other hand for small angular increments with two angles we obtain quite informative reconstructions, as displayed in \Cref{fig:stones2Ang}. The reconstruction results for the tracking setting can be seen in \Cref{fig:stonesBegEnd} and the corresponding reconstructions for the randomized angles are presented in \Cref{fig:stones30Rand}.

\begin{figure}[h!]
\centering
\begin{picture}(450,280)
\put(0,145){\includegraphics[width=.32\textwidth]{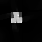}}
\put(150,145){\includegraphics[width=.32\textwidth]{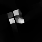}}
\put(300,145){\includegraphics[width=.32\textwidth]{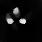}}
\put(0,-5){\includegraphics[width=.32\textwidth]{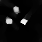}}
\put(150,-5){\includegraphics[width=.32\textwidth]{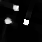}}
\put(300,-5){\includegraphics[width=.32\textwidth]{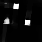}}
\end{picture}	\caption{$L^1$-$TV$ reconstruction result for the Rolling Stones from small angluar increments with two angles per time step. From top left to bottom right: time steps 1, 7, 13, 18, 25, 30.}
	\label{fig:stones2Ang}
\end{figure}

\begin{figure}[t!]
\centering
\begin{picture}(450,280)
\put(0,145){\includegraphics[width=.32\textwidth]{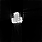}}
\put(150,145){\includegraphics[width=.32\textwidth]{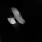}}
\put(300,145){\includegraphics[width=.32\textwidth]{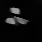}}
\put(0,-5){\includegraphics[width=.32\textwidth]{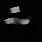}}
\put(150,-5){\includegraphics[width=.32\textwidth]{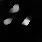}}
\put(300,-5){\includegraphics[width=.32\textwidth]{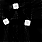}}
\end{picture}	\caption{$L^1$-$TV$ reconstruction result for the Rolling Stones from tracking. From top left to bottom right: time steps 1, 7, 13, 18, 25, 30.}
	\label{fig:stonesBegEnd}
\end{figure}

\begin{figure}[t!]
\centering
\begin{picture}(450,280)
\put(0,145){\includegraphics[width=.32\textwidth]{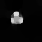}}
\put(150,145){\includegraphics[width=.32\textwidth]{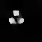}}
\put(300,145){\includegraphics[width=.32\textwidth]{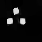}}
\put(0,-5){\includegraphics[width=.32\textwidth]{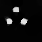}}
\put(150,-5){\includegraphics[width=.32\textwidth]{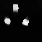}}
\put(300,-5){\includegraphics[width=.32\textwidth]{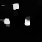}}
\end{picture}	\caption{$L^1$-$TV$ reconstruction result for the Rolling Stones from one randomized per time step. From top left to bottom right: time steps 1, 7, 13, 18, 25, 30.}
	\label{fig:stones30Rand}
\end{figure}

\begin{figure}[h!]
\centering
\begin{picture}(450,275)
\put(0,140){\includegraphics[width=.32\textwidth]{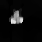}}
\put(150,140){\includegraphics[width=.32\textwidth]{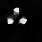}}
\put(300,140){\includegraphics[width=.32\textwidth]{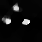}}
\put(0,-10){\includegraphics[width=.32\textwidth]{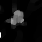}}
\put(150,-10){\includegraphics[width=.32\textwidth]{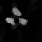}}
\put(300,-10){\includegraphics[width=.32\textwidth]{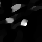}}
\end{picture}
	\caption{$L^1$-$TV$ reconstruction result for Rolling Stones data set calculated from one angle per time in a random arrangement. Top: 15 total time steps, bottom: 8 total time steps. From left to right: time steps 1, 13, 25 (taken from the full 30 time steps).}
	\label{fig:stonesLessAng}
\end{figure}

For all three measurement setups presented we can reconstruct the position of the stones quite clearly. Especially in case of small angular increments with two angles and randomized angles, we can separate the stones already from the beginning and clearly track the movement. For the tracking approach the initial and end states are clearly reconstructed due to the full angle data, but the position of the stones during the movement can only be identified after they have separated sufficiently.

The randomized angles provide a superior reconstruction quality with respect to the amount of used data. The motivation of this study is to reduce the amount of necessary measurements for dynamic data as far as possible and hence the randomized angles are most successful. We point out that we only have one projection image per time step and hence we are not able to reduce the data any further. However, we try to lower the amount of used projections even more by reducing the time steps and hence increasing the spacial offset between frames.
In \Cref{fig:stonesLessAng} reconstructions for a total of 15 and 8 time steps are presented. For 15 projections, the separation as well as the position of the stones are still well reconstructed, just a slightly stronger blurring occurs due larger movements between frames.
Regarding the results calculated from only 8 projections, the blurring has strongly increased and the shape of the stones is not clear anymore. This could be considered as the limit of our approach to produce reasonable results.

\newpage 
\section{Conclusions and outlook}\label{sec:conclusions}  
We introduced a framework to combine motion estimation and reconstruction in X-ray tomography in a joint model. The aim of this study is to illustrate that one can estimate the motion from the measured data in a variational framework in order to reconstruct the dynamics of the measured object in space-time. For estimating the motion we utilized the optical flow framework, which is already well-established in image registration, but has been only recently introduced to inverse problems. The forward problem in our framework is modeled by a time-dependent Radon transform that is well-defined in a finite time setting. We then combined the reconstruction task and the motion estimation to a joint model, which can be solved in an alternating way with modern optimization techniques. Additionally, we propose a probabilistic perspective on error modeling that will be studied further in future research.

The proposed model has been applied to simulated and real phantoms with extremely undersampled data, i.e. we went as low as one projection per time step and yet were able to produce informative results capturing the dynamics of the system correctly. The experiments showed that an $L^1$-$TV$ model for the reconstruction is most powerful for this kind of data. Furthermore, we obtained the best results with randomly chosen projection angles in each time instance. This will be of special interest for further studies in the context of compressed sensing in X-ray tomography in particular for dynamic systems.

Optical flow respectively the continuity equation as motion model is rather restrictive in the context of some inverse problems in medical imaging, since one cannot introduce new mass to the system, e.g. in terms of tracers or signals. Hence it is important to investigate different motion models. A first extension that comes to ones mind is to allow input to the system by considering non-zero Neumann boundary conditions. Relevant applications include cardiac scans, where a tracer is injected to the patients blood stream to monitor blood flow through the heart. Many imaging modalities also include diffusion processes and, therefore, are not suitable for the optical flow model. We leave this and further extensions to future research.

\section*{Acknowledgments}
This work has been supported by the German Science Exchange Foundation DAAD via Project 57162894, Bayesian Inverse Problems in Banach Space, as well as by the Academy of Finland through the Finnish Centre of Excellence in Inverse Problems Research 2012--2017, decision number
250215. MB, HD and LF acknowledge further support by ERC via Grant EU FP 7 - ERC Consolidator Grant 615216 LifeInverse. AH was partially supported by FiDiPro project of the Academy of Finland, decision number 263235. TH was supported by Academy of Finland via project 275177.

\bibliographystyle{siam}
\bibliography{Inverse_problems_references_2016}

\end{document}